\documentclass{amsart}
  \usepackage{amsmath,amsfonts}
  \usepackage{amssymb,enumerate}
  \usepackage{pdfsync,verbatim}
  \usepackage{color}  %%so that I can color my comments 

 \newcommand{\I}{\mathrm{i}}
\newcommand{\D}{\mathrm{d}}

\newcommand{\lb}{\left(}

\newcommand{\vp}{\varphi}

\newcommand{\rb}{\right)}
\newcommand{\PD}{\partial}

\newcommand{\wt}{\widetilde}

\newcommand{\Fc}{\mathcal{F}}

\newcommand{\Rb}{\mathbb{R}}

\newcommand{\Beq}{\begin{equation}}
\newcommand{\Eeq}{\end{equation}}
\newcommand{\beq}{\begin{equation*}}
\newcommand{\eeq}{\end{equation*}}
\newcommand{\bal}{\begin{align}}
\newcommand{\eal}{\end{align}}

\newcommand{\g}{\gamma}

\newcommand{\A}{\alpha}
\newcommand{\B}{\beta}
\newcommand{\bp}{\begin{prob}}
\newcommand{\ep}{\end{prob}}
\newcommand{\bpr}{\begin{proof}}
\newcommand{\epr}{\end{proof}}
\newcommand{\st}{\,:\,} 
\newcommand{\rank}{\operatorname{rank}}

\newcommand{\bel}[1]{\begin{equation}\label{#1}}
\newcommand{\ee}{\end{equation}} 
 
\newcommand{\norm}[1]{|{#1}|}

\newcommand{\rtwo}{\mathbb{R}^2}
\newcommand{\rr}{\mathbb{R}}

\newcommand{\NT}{\negthinspace}
\newtheorem{theorem}{Theorem}[section]

\newtheorem{lemma}[theorem]{Lemma}
\newtheorem{proposition}[theorem]{Proposition}

\newtheorem{example}[theorem]{Example}

\theoremstyle{definition}
\newtheorem{definition}[theorem]{Definition}

\title[Bistatic SAR Imaging]
      {Microlocal Aspects of Bistatic Synthetic Aperture Radar Imaging}

\author[Venky P. Krishnan and Eric Todd Quinto]{}


\thanks{* Both authors were partially supported by NSF grant DMS
0908015.}

\begin{document}

\begin{abstract}
 In this article, we analyze the microlocal properties of the  linearized forward scattering operator $F$ and the reconstruction operator $F^{*}F$ appearing in bistatic synthetic aperture radar imaging.
In our model, the radar source and detector travel along a line a fixed
distance apart. We show that $F$ is a Fourier integral operator, and
we give the mapping properties of the projections from the canonical
relation of $F$, showing that the right projection is a blow-down and the left projection is a fold.  We then 
show that $F^{*}F$ is a singular FIO belonging to the class $I^{3,0}$.
\end{abstract}

\maketitle

%% Enter the first author's name and address:
\centerline{\scshape Venky P. Krishnan*  and Eric Todd Quinto*}
\medskip
{\footnotesize
 %% please put the address of the first author
 \centerline{Department of Mathematics, Tufts University}
   \centerline{503, Boston Avenue, Medford, MA 02155, USA}
} %% Do not forget to end the {\footnotesize by the sign }

%\medskip

%\centerline{\scshape Eric Todd Quinto }
%\medskip
%{\footnotesize
 %% please put the address of the second author
% \centerline{ Department of Mathematics, Tufts University}
%   \centerline{503, Boston Avenue, Medford, MA 02155, USA}
%} %

\bigskip

%% The name of the associate editor will be entered by an editorial staff
% \centerline{(Communicated by the associate editor name)}

\maketitle

\section{Introduction} \label{sect:Introduction}

In synthetic aperture radar (SAR) imaging, a region of interest on the
surface of the earth is illuminated by electromagnetic waves from a
moving airborne platform. The goal is to reconstruct an image of the
region based on the measurement of scattered waves. For an in-depth
treatment of SAR imaging, we refer the reader to
\cite{CheneyBordenBook, Cheney2001}. SAR imaging is similar to other
imaging problems such as Sonar where acoustic waves are used to
reconstruct the shape of objects on the ocean floor \cite{And, CB1979,
LQ2000}.

%The transmitters are usually expensive and can be flown on a platform different from that of the receivers which are inexpensive. Also 

In monostatic SAR, the source and the receiver are located on the
same moving airborne platform. In bistatic SAR, the source and the
receiver are on independently moving airborne platforms.  There are
several advantages to considering such data acquisition geometries.
The receivers, compared to the transmitters, are passive and hence are
more difficult to detect. Hence, by separating their locations,
the receivers alone can be in an unsafe environment, while the
transmitters are in a safe environment. Furthermore, bistatic SAR
systems are more resistant to electronic countermeasures such as
target shaping to reduce scattering in the direction of incident waves
\cite{HorneYates}.  

In this paper, we consider a bistatic SAR  system where the antennas have poor directivity and hence the beams do not focus on targets on the ground. We assume that the transmitter and receiver traverse a 1-dimensional curve and the back-scattered data is measured at each point on this curve for a certain period of time. As in the monostatic SAR case \cite{NC2002}, with a weak scattering assumption, the linear scattering operator that relates the unknown function that models the object on the ground to the data at the receiver (see \cite{YarmanYaziciCheney}) is a Fourier integral operator (FIO) \cite{Ho1971,Ho, Duistermaat,Treves}.  Now if $\Fc$ is an FIO, the canonical relation $\Lambda_{\Fc}$ associated to $\Fc$ tells us how the singularities of the object are propagated to the data. The canonical relation $\Lambda_{\Fc^{*}}$ of the $L^{2}$ adjoint $\Fc^{*}$ of $\Fc$ gives us information as to how the singularities in the data are propagated back to the reconstructed object. The microlocal analysis of singularities of the object is then done by analyzing the composition  $\Lambda_{\Fc^{*}}\circ \Lambda_{\Fc}$.

Such an analysis for monostatic SAR has been done by several authors \cite{NC2004, Felea,RF2} and  is fairly well understood. In their work \cite{NC2004}, Nolan and Cheney showed that the composition of the linearized scattering operator with its $L^{2}$ adjoint is a singular pseudodifferential operator ($\Psi$DO) and it belongs to class of Fourier
integral operators with two cleanly intersecting Lagrangians. Felea in her works \cite{Felea,RF2} further analyzed the properties of the composition of these operators. 

In this paper, we do a similar analysis for the bistatic SAR imaging problem. Given the complications that arise in treating arbitrary transmitter and receiver trajectories, in this paper, we focus on the case where the transmitter and receiver are at the same height above the ground, traverse the same linear trajectory at the same constant speed and spaced apart from each other by a constant distance. Furthermore, we assume that the object to be imaged is on the ground, which for simplicity, we will assume is flat. Since the measured data is two-dimensional, it is reasonable to expect that we can reconstruct a two-dimensional object.

The outline of the paper is as follows. Section
\ref{sect:Preliminaries} focuses on the preliminaries. Here we give
the linearized scattering model for bistatic SAR and definitions of
singularities and distributions that belong to $I^{p,l}$ classes and
important results on distributions belonging to this class that are
required in this paper. In Section \ref{sect:ForwardModel}, we
undertake a detailed study of the canonical relation associated this
FIO. This is the content of Theorem \ref{Theorem-1}. Then we study the
 reconstruction operator, that is, the composition of the
bistatic scattering operator with its adjoint, and show in Theorem
\ref{Theorem-2} that this operator belongs to the class
$I^{3,0}$. Our proof of \ref{Theorem-2} follows the ideas of
\cite[Theorem 1.6]{Felea}. Several identities required to prove
Theorem \ref{Theorem-2} are provided in the Appendix. 

\section{Preliminaries}\label{sect:Preliminaries}

\subsection{The bistatic linearized scattering model} We assume that a
bistatic SAR system is involved in imaging a scene. Let $\g_{T}(s)$
and $\g_{R}(s)$ for $s\in (s_{0},s_{1})$ be the trajectories of the
transmitter and receiver respectively. The transmitter transmits
electromagnetic waves that scatter off the target, which are then
measured at the receiver. We are interested in obtaining a linearized
model for this scattered signal. 

The propagation of electromagnetic waves can be described by the
scalar wave equation:
\begin{equation*}\label{Scalar Wave Equation}
\lb \Delta-\frac{1}{c^{2}}\PD_{t}^{2}\rb E(x,t)=P(t)\delta(x-\g_{T}(s)),
\end{equation*}
where $c$ is the speed of electromagnetic waves in the medium, $E(x,t)$ is the electric field
and $P(t)$ is the transmit waveform sent to the transmitter antenna located at position $\g_{T}(s)$. The wave speed $c$ is spatially varying due to inhomogeneities present in the medium.
We assume that the background in which the electromagnetic waves propagate is free space. Therefore $c$ can be expressed as:
\begin{equation*}
\frac{1}{c^{2}(x)}=\frac{1}{c_{0}^{2}}+\wt{V}(x),
\end{equation*} 
where the constant $c_{0}$ is the speed of light in free space and $\wt{V}(x)$ is the perturbation due to deviation from the background, which we would like to recover from backscattered waves. 

Since the incident electromagnetic waves in typical radar frequencies attenuate rapidly as they penetrate the ground, we assume that $\wt{V}(x)$  varies only on a 2-dimensional surface. Therefore, we represent $\wt{V}$ as a function of the form
\begin{equation*}
\wt{V}(x)=V( x)\delta_{0}(x_{3})
\end{equation*}
where we assume for simplicity that the earth's surface is flat, represented by the $x=(x_{1},x_{2})$ plane.

The background Green's function $g$ is then given by the solution to the following equation:
\begin{align*}
\lb \Delta-\frac{1}{c_{0}^{2}}\PD_{t}^{2}\rb g(x,\g_{T}(s),t)=\delta(x-\g_{T}(s))\delta(t).
\end{align*}
We can explicitly write $g$ as 
\begin{equation*}\label{Background Green's function}
g(x,\g_{T}(s),t)=\frac{\delta(t-|x-\g_{T}(s)|/c_{0})}{4\pi |x-\g_{T}(s)|}. 
\end{equation*}
Now the incident field $E^{\mathrm{in}}$ due to the source  $s(x,t)=P(t)\delta(x-\g_{T}(s))$ is 
\begin{align*}\label{Incident Field}
E^{\mathrm{in}}( x,t)&=\int g( x, y,t-\tau)s( y,\tau)\D  y\D \tau\\
&=\frac{P(t-| x-\g_{T}(s)|/c_{0})}{4\pi | x-\g_{T}(s)|}. 
\end{align*}
Let $E$ denote the total field of the medium, $E=E^{\mathrm{in}}+E^{\mathrm{sc}}$. Then the scattered field can be written using the Lippman-Schwinger equation:
\begin{equation}\label{Lippman-Schwinger}
E^{\mathrm{sc}}( z,t)=\int g( z, x,t-\tau)\PD_{t}^{2}E( x,\tau)T( x)\D  x\D \tau.
\end{equation}
We linearize \eqref{Lippman-Schwinger} by the first Born approximation and write the linearized scattered wave-field at receiver location $\g_{R}(s)$:
\begin{align}
\notag E^{\mathrm{sc}}_{\mathrm{lin}}(\g_{R}(s),t)&= \int g(\g_{R}(s), x,t-\tau)\PD_{t}^{2}E^{\mathrm{in}}( x,\tau)V( x)\D  x\D \tau\\
\label{Scattered-wave-field-2}&=\int \frac{\delta(t-\tau-| x-\g_{R}(s)|/c_{0})}{4\pi| x-\g_{R}(s)|} \left( e^{-\I \omega(\tau-| x-\g_{T}(s)|/c_{0})}\frac{\omega^{2}p(\omega)}{4\pi| x-\g_{T}(s)|}\right)\\
\notag &\hspace{0.2in}\times V( x)\D \omega\D  x\D \tau,
\end{align}
where $p$ is the Fourier transform of $P$.

Now, integrating \eqref{Scattered-wave-field-2} with respect to $\tau$, a linearized model for the scattered signal is as follows:
\begin{equation}\label{Bi-static Mathematical Model}
d(s,t):=E^{\mathrm{sc}}_{\mathrm{lin}}(\g_{R}(s),t)=\int e^{-\I \omega(t-\frac{1}{c_{0}}R(s,x))}A(s,x,\omega)V(x)\D x\D \omega,
\end{equation}
where 
\begin{equation}\label{Bi-static distance}
R(s,x)=|\g_{T}(s)-x|+|x-\g_{R}(s)|
\end{equation}
and 
\begin{equation}\label{Bi-static Amplitude}
A(s,x,\omega)=p(\omega)((4\pi)^{2}|\g_{T}(s)-x||\g_{R}(s)-x|)^{-1}.
\end{equation}
This function includes terms that take into account the transmitter and receiver antenna beam patterns, the transmitted waveform and geometric spreading factors.

We will show in Section \ref{sect:ForwardModel} in one important case,
that the transform $\Fc$ that maps $V$ to \eqref{Bi-static
Mathematical Model} is a Fourier integral operator
associated to a canonical relation $\Lambda$ (Proposition \ref{prop:C}),
and we will prove the mapping properties of $\Lambda$ (Propositions
\ref{St-line left projection} and \ref{prop:piR}).  These mapping
properties tell what $\Fc$ does to singularities.  We now define the
mapping properties we need.

\subsection{Singularities and $I^{p,l}$ classes}

Here we give the definitions of the singularities associated with our operator
$\Fc$ and its canonical relation \eqref{Bi-static Twisted Canonical
Relation}, and a class of distributions required for the analysis of the composition $\Fc$ with its $L^{2}$ adjoint.

 \begin{definition}\label{def:fold-blowdown} \cite{GU1990b} Let $M$
and $N$ be manifolds of dimension $n$ and let $f:N\to M$ be
$C^\infty$.

\begin{enumerate}

\item $f$ is a \emph{Whitney fold} if near each $m_0\in
M$, $f$ is either a local diffeomorphism or $f$ drops
rank  simply by one at $m_0$ so that $L= \{m\in M\st \rank df = n-1\}$
is a smooth hypersurface at $m_0$ and  $\ker df(m_0)\not\subset
T_{m_0}L$.

\item $f$ is a \emph{blow-down} along a smooth hypersurface $L \subset
M$ if $f$ is a local diffeomorphism away from $L$ and $f$ drops rank
simply by one at $L$, where the Hessian of $f$ is equal to zero and
$\ker df \subset T(L)$, so that $f\big|_L$ has one-dimensional fibers.

\end{enumerate}
\end{definition}

We now define $I^{p,l}$ classes. They were first introduced by Melrose
and Uhlmann, \cite{MU} Guillemin and Uhlmann \cite{GUh} and Greenleaf
and Uhlmann \cite{GU1990a} and they were used in the context of radar
imaging in \cite{NC2004,Felea, RF2}.

\begin{definition} Two submanifolds  $M$ and $N$ intersect {\it cleanly} if $M \cap N$ is a smooth submanifold and $T(M \cap N)=TM \cap TN$.
\end{definition}

Consider two spaces $X$ and $Y$ and let $\Lambda_0$ and $\Lambda_1$ and
$\tilde{\Lambda}_0$ and $\tilde{\Lambda}_1$ be Lagrangian submanifolds of
the product space $T^*X \times T^*Y$.  If they intersect cleanly
$(\tilde{\Lambda}_0, \tilde{\Lambda}_1)$ and $(\Lambda_0, \Lambda_1)$ are
equivalent in the sense that there is, microlocally, a canonical
transformation $\chi$ which maps $(\Lambda_0, \Lambda_1)$ into
$(\tilde{\Lambda}_0, \tilde{\Lambda}_1)$.  This leads us to the following
model case.

\begin{example}\label{ex:model}  Let  $\tilde{\Lambda}_0=\Delta_{T^*R^n}=\{ (x, \xi; x,
\xi)| x \in R^n, \ \xi \in R^n \setminus 0 \}$ be  the diagonal in
$T^*R^n \times T^*R^n$ and let  $\tilde{\Lambda}_1= \{ (x', x_n, \xi', 0; x',
y_n, \xi', 0)| x' \in R^{n-1}, \ \xi' \in R^{n-1} \setminus 0 \}$. 
Then, $\tilde{\Lambda}_0$ intersects $ \tilde{\Lambda}_1$ cleanly in
codimension $1$.\end{example}

Now we define the class of product-type symbols $S^{p,l}(m,n,k)$.

\begin{definition}
$S^{p,l}(m,n,k)$ is the set of all functions $a(z,\xi,\sigma) \in
C^{\infty} ({\bf R}^m \times {\bf R}^n \times {\bf R}^k )$ such that
for every $K \subset {\bf R}^m$ and every $\alpha \in {\bf Z}^n_+,
\beta \in {\bf Z}^n_+, \gamma \in {\bf Z}^k_+$ there is $c_{K, \alpha,
\beta}$ such
that
\[|\partial_z^{\alpha}\partial_{\xi}^{\beta}\partial_{\sigma}^{\gamma}
a(z,\xi,\sigma)| \le c_{K,\alpha,\beta}(1+ |\xi|)^{p- |\beta|} (1+
|\sigma|)^{l-| \gamma|}, \forall (z,\xi,\tau) \in K \times {\bf R}^n
\times {\bf R}^k.\]
\end{definition}

Since any two sets of cleanly intersecting Lagrangians are equivalent,
we first define $I^{p,l}$ classes for the case in Example
\ref{ex:model}.

\begin{definition}\cite{GUh}    Let $I^{p,l}(\tilde{\Lambda}_0,
\tilde{\Lambda}_1)$  be the set of all distributions  $u$
such that $u=u_1 + u_2$  with $u_1 \in C^{\infty}_0$
 and $$u_2(x,y)=\int e^{i((x'-y'-s)\cdot \xi'+(x_n-y_n-s) \cdot \xi_n+ s
\cdot \sigma)} a(x,y,s; \xi,\sigma)d\xi d\sigma ds$$  with $a \in
S^{p',l'}$  where $p'=p-\frac{n}{2}+\frac{1}{2}$  and
$l'=l-\frac{1}{2}$.
\end{definition}

This allows us to define the $I^{p,l}(\Lambda_0, \Lambda_1)$ class for any
two cleanly intersecting Lagrangians in codimension $1$ using the
microlocal equivalence with the case in Example \ref{ex:model}.

\begin{definition}   \cite{GUh} Let  $I^{p,l}( \Lambda_0, \Lambda_1)$
be the set of all distributions $u$ such that $ u=u_1 + u_2 + \sum v_i$
where $u_1 \in I^{p+l}(\Lambda_0 \setminus \Lambda_1)$, $u_2 \in
I^{p}(\Lambda_1 \setminus \Lambda_0)$, the sum $\sum v_i$ is locally
finite and $v_i=Fw_i$ where $F$ is a zero order FIO associated to
$\chi ^{-1}$, the canonical transformation from above, and $w_i \in
I^{p,l}(\tilde {\Lambda}_0, \tilde{\Lambda}_1)$.
\end{definition}

This class of distributions is invariant under FIOs associated to
canonical transformations which map the pair $(\Lambda_0, \Lambda_1)$
to itself.  By definition, $F\in I^{p,l}(\Lambda_{0},\Lambda_{1})$ if
its Schwartz kernel belongs to $I^{p,l}(\Lambda_{0},\Lambda_{1})$. If
$F \in I^{p,l}(\Lambda_0, \Lambda_1)$ then $F \in I^{p+l}(\Lambda_0
\setminus \Lambda_1)$ and $F \in I^p(\Lambda_1 \setminus \Lambda_0)$
\cite{GUh}.  Here by $F\in I^{p+l}(\Lambda_{0} \setminus
\Lambda_{1})$, we mean that the Schwartz kernel of $F$ belongs to
$I^{p+l}(\Lambda_{0}\setminus \Lambda_{1})$ microlocally away from
$\Lambda_{1}$.

\section{Transmitter and receiver in a linear
trajectory}\label{sect:ForwardModel}
Henceforth, let us assume that the trajectory of the transmitter is 
\[\g_{T}: (s_{0},s_{1})\to \Rb^{3},\quad \g_{T}(s)=(s+\A,0,h)\] and 
the trajectory of the receiver is \[\g_{R}(s): (s_{0},s_{1})\to
\Rb^{3}, \quad \g_{R}(s)=(s-\A,0,h).\] Here $\A>0$ and $h>0$ are fixed.
From Equation \eqref{Bi-static Mathematical Model}, the linearized 
model for the data at the receiver, for $s\in (s_{0},s_{1})$ and 
$t\in(t_{0},t_{1})$ is
\begin{equation}\label{Bi-static data}
d(s,t)=\int e^{-\I \omega\lb
t-\frac{1}{c_{0}}\lb|x-\g_{T}(s)|+|x-\g_{R}(s)|\rb\rb}A(s,x,\omega)V(x)\D
x\D \omega.
\end{equation}
We multiply $d(s,t)$ by a smooth (infinitely differentiable) function $f(s,t)$ supported in a compact subset of $(s_{0},s_{1})\times (t_{0},t_{1})$. This compensates for the discontinuities in the measurements at the end points of the rectangle $(s_{0},s_{1})\times (t_{0},t_{1})$. For simplicity, let us denote the function $f\cdot d$ as $d$ again. We then have 
\begin{equation}\label{Modified data-1}
d(s,t)=\int e^{-\I \omega(t-\frac{1}{c_{0}}R(s,x))}A(s,t,x,\omega)V(x)\D x\D \omega,
\end{equation}
where now $A(s,t,x,\omega)=f(s,t)A(s,x,\omega)$. 

Our method cannot image the point on the object that is ``directly underneath'' the transmitter and the receiver. This is, if the transmitter and receiver are at locations $(s+\A,0,h)$ and $(s-\A,0,h)$, then we cannot image the point $(s,0,0)$. Therefore we modify $d$ in Equation \eqref{Modified data-1} by multiplying by another smooth function $g(s,t)$ such that 
\[
g\equiv 0 \quad \mbox{ in a small neighborhood of } \Bigg{\{}\lb s,2\frac{\sqrt{\A^{2}+h^{2}}}{c_{0}}\rb \st s_{0}<s<s_{1}\Bigg{\}}.
\]

For simplicity, again denote $g\cdot d$ as $d$ and $g\cdot A$ as $A$. Consider,
\begin{equation}\label{Bi-static FIO}
F V(s,t):=d(s,t)=\int e^{-\I \omega(t-\frac{1}{c_{0}}\lb |x-\g_{T}(s)|+|x-\g_{R}(s)|\rb)}A(s,t,x,\omega)V(x)\D x\D \omega.
\end{equation}
For simplicity, let us denote the $(s,t)$ space as $Y$.

We assume that the amplitude function $A$ satisfies the following estimate: For every compact $K\in Y\times X$ and for every non-negative integer $\A$ and for every 2-indexes $\B=(\B_{1},\B_{2})$ and $\g$, there is a constant $C$ such that 
\begin{equation}\label{Amplitude Estimate}
|\PD_{\omega}^{\A}\PD_{s}^{\B_{1}}\PD_{t}^{\B_{2}}\PD_{x}^{\g}A(s,t,x,\omega)|\leq C(1+|\omega|)^{2-\A}.
\end{equation}

The phase function of the operator $F$,
\bel{def:psi}\psi(s,t,x,\omega)=-\omega\left(t-\frac{1}{c_{0}}(|x-\g_{T}(s)|+|x-\g_{R}(s)|)\right)\ee
is homogeneous of degree $1$  in $\omega$.

We now analyze some properties of the canonical relation of the
operator $F$.

\begin{proposition}\label{prop:C} $F$ is a Fourier integral operator
of order 3/2
associated with canonical relation 
\begin{align}\label{Bi-static Twisted Canonical Relation}
\Lambda =&\Bigg{\{}\lb s,t,-\frac{\omega}{c_{0}}\lb
\frac{x_{1}-s-\A}{|x-\g_{T}(s)|}+\frac{x_{1}-s+\A}{|x-\g_{R}(s)}\rb,-\omega\rb;\\
\nonumber & \Bigg{(}\NT x_{1},x_{2},-\frac{\omega}{c_{0}}\NT\lb \frac{x_{1}-s-\A}{|x-\g_{T}(s)|}+\frac{x_{1}-s-\A}{|x-\g_{R}(s)|}\rb,-\frac{\omega}{c_{0}}\NT\lb \frac{x_{2}}{|x-\g_{T}(s)|}+\frac{x_{2}}{|x-\g_{R}(s)|}\rb\negthinspace\negthinspace\Bigg{)}\\
\nonumber &: c_{0}t=\sqrt{(x_{1}-s-\A)^{2}+x_{2}^{2}+h^{2}}+\sqrt{(x_{1}-s+\A)^{2}+x_{2}^{2}+h^{2}},\quad \omega\neq 0 \Bigg{\}}.
\end{align}  Furthermore $(x_{1},x_{2},s,\omega)$ is a global
parameterization for  $\Lambda$.\end{proposition}

\begin{proof}  This is a straightforward application of the theory of
FIO.  Since $\psi$ in  \eqref{def:psi} is a nondegenerate phase function
with $\PD_{x}\psi$ and $\PD_{s,t}\psi$ nowhere zero and the amplitude
$A$ in \eqref{Bi-static FIO} is of order 2, $F$ is an FIO
\cite{Ho1971}.  Since the amplitude is of order $2$, the order of the FIO is 3/2 by \cite[Definition 3.2.2]{Ho1971}.  By
definition \cite[Equation (3.1.2)]{Ho1971} 
\[
\Lambda =
\{(s,t,\PD_{s,t}\psi(x,s,t,\omega)),(x,-\PD_{x}\psi(x,s,t)) \st
\PD_{\omega}\psi(x,s,t,\omega)=0\}.\]  A calculation using this definition
establishes \eqref{Bi-static Twisted Canonical Relation}. Finally, it is easy to see that $(x_{1},x_{2},s,\omega)$ is a global parameterization of $\Lambda$.

\end{proof} 

  In order to understand the microlocal mapping properties of $F$ and
$F^{*}F$, we consider the projections $\pi_{L}: T^{*}Y\times
T^{*}X\to T^{*}Y$ and $\pi_{R}:T^{*}Y\times T^{*}X\to T^{*}X$.

\begin{proposition}\label{St-line left projection}
The projection $\pi_{L}$ restricted to $\Lambda$ has a fold singularity on
$\Sigma:=\{(x,0,s,\omega):\omega\neq 0\}$.
\end{proposition}

\begin{proof}
The projection $\pi_{L}$ is given by
\begin{equation}\label{Left Projection}
\pi_{L}(x_{1},x_{2},s,\omega)=
\end{equation}
\begin{align*}
=\Bigg{(} s,\frac{1}{c_{0}}\lb |x-\g_{T}(s)|+|x-\g_{R}(s)|\rb,
-\frac{\omega}{c_{0}}\lb \frac{x_{1}-s-\A}{|x-\g_{T}(s)|}+\frac{x_{1}-s+\A}{|x-\g_{R}(s)}\rb,-\omega\Bigg{)}
\end{align*}
We have 
\begin{equation*}
\nonumber(\pi_{L})_{*}=
\left(
\begin{matrix}
0&0&1&0\\
\frac{1}{c_{0}}\lb \frac{x_{1}-s-\A}{|x-\g_{T}(s)|}+\frac{x_{1}-s+\A}{|x-\g_{R}(s)|}\rb&\frac{1}{c_{0}}\lb \frac{x_{2}}{|x-\g_{T}(s)|}+\frac{x_{2}}{|x-\g_{R}(s)|}\rb&*&0\\
-\frac{\omega}{c_{0}}\lb \frac{x_{2}^{2}+h^{2}}{|x-\g_{T}(s)|^{3}}+\frac{x_{2}^{2}+h^{2}}{|x-\g_{R}(s)|^{3}}\rb& \frac{\omega}{c_{0}}\lb \frac{(x_{1}-s-\A)x_{2}}{|x-\g_{T}(s)|^{3}}+\frac{(x_{1}-s+\A)x_{2}}{|x-\g_{R}(s)|^{3}}\rb&*&*\\
0&0&0&-1
\end{matrix}
\right)
\end{equation*}
Then 
\[
\det(\pi_{L})_{*}=\frac{\omega}{c_{0}^{2}}x_{2}\lb
\frac{1}{|x-\g_{T}(s)|^{2}}+\frac{1}{|x-\g_{R}(s)|^{2}}\rb\lb
1+\frac{(x_{1}-s)^{2}+x_{2}^{2}+h^{2}-\A^{2}}{|x-\g_{T}(s)||x-\g_{R}(s)|}\rb.
\]
The proposition then follows as a consequence of the following lemma.
\begin{lemma}\label{St-line-nonvanishing lemma}
The term
\[
1+ \frac{(x_{1}-s)^{2}+x_{2}^{2}+h^{2}-\A^{2}}{|x-\g_{T}(s)||x-\g_{R}(s)|}.
\]
is positive for all $x\in \rtwo$, $s\in \rr$ and $h$ and $\alpha$
positive.
\end{lemma}
\bpr This is a straightforward calculation that is made simpler
if one lets $S=x_1-s$, $T = \sqrt{x_2^2+h^2}$ then puts the term over
a common denominator.  Finally, one shows that the numerator is
positive by isolating the square roots (absolute values), squaring,
and simplifying to infer that, since $4T^2\alpha^2>0$, the lemma is
true. 
\epr

Now returning to the proof of the Proposition \ref{St-line left projection}, we have that $\det(\pi_{L})_{*}=0$ if and only if $x_{2}=0$. Hence $\det(\pi_{L})_{*}$ vanishes on the set $\Sigma$ and Lemma \ref{St-line-nonvanishing lemma} again shows that $\D(\det(\pi_{L})_{*})$ on $\Sigma$ is non-vanishing. This implies that $\pi_{L}$ drops rank simply by one on $\Sigma$.

Now it remains to show that $T\Sigma\cap \mbox{Kernel}(\pi_{L})_{*}=\{0\}$. But this follows from the fact that $\mbox{Kernel}(\pi_{L})_{*}=\mbox{span}(\frac{\PD}{\PD x_{2}})$, but $T\Sigma=\mbox{span}(\frac{\PD}{\PD x_{1}},\frac{\PD}{\PD s},\frac{\PD}{\PD \omega})$. This concludes the proof of the proposition.
\end{proof}

\begin{proposition}\label{prop:piR}
 Consider the projection $\pi_{R}:T^{*}Y\times T^{*}X\to T^{*}X$. The restriction of the projection to $\Lambda$ has a blowdown singularity on $\Sigma$.
\end{proposition}
\begin{proof}
We have
\begin{equation}\label{Right Projection}
\pi_{R}(x_{1},x_{2},s,\omega)=
\end{equation}
\begin{align*}
=\Bigg{(}x_{1},x_{2},-\frac{\omega}{c_{0}}\lb \frac{x_{1}-s-\A}{|x-\g_{T}(s)|}+\frac{x_{1}-s-\A}{|x-\g_{R}(s)|}\rb,
-\frac{\omega}{c_{0}}\lb \frac{x_{2}}{|x-\g_{T}(s)|}+\frac{x_{2}}{|x-\g_{R}(s)|}\rb\Bigg{)}.
\end{align*}

Now 
\begin{equation}
\nonumber (\pi_{R})_{*}=
\left(
\begin{matrix}
 1&0&0&0\\
0&1&0&0\\
*&*&-\frac{\omega}{c_{0}}\lb
\frac{x_{2}^{2}+h^{2}}{|x-\g_{T}(s)|^{3}}+\frac{x_{2}^{2}+h^{2}}{|x-\g_{R}(s)|^{3}}\rb&
-\frac{1}{c_{0}}\lb \frac{x_{1}-s-\A}{|x-\g_{T}(s)|}+\frac{x_{1}-s+\A}{|x-\g_{R}(s)|}\rb\\
*&*&\frac{\omega}{c_{0}}\lb
\frac{(x_{1}-s-\A)x_{2}}{|x-\g_{T}(s)|^{3}}+\frac{(x_{1}-s+\A)x_{2}}{|x-\g_{R}(s)|^{3}}\rb&
-\frac{1}{c_{0}}\lb
\frac{x_{2}}{|x-\g_{T}(s)|}+\frac{x_{2}}{|x-\g_{R}(s)|}\rb
\end{matrix}
\right)
\end{equation}
From this we see that $\mbox{Kernel}(\pi_{R})_{*} \subset T\Sigma$. Since $\det((\pi_{R})_{*})=\det((\pi_{L})_{*})$, $\pi_{R}$ drops rank simply by one along $\Sigma$. Therefore the projection $\pi_{R}$ has a blowdown singularity along $\Sigma$.
\end{proof}
 We summarize what we have proved in this section by the following theorem:
\begin{theorem}\label{Theorem-1}
The operator $F$ defined in \eqref{Bi-static FIO} is a Fourier integral operator of order $3/2$. The canonical relation $\Lambda$ associated to $F$ defined in \eqref{Bi-static Twisted Canonical Relation} satisfies the following: The projections $\pi_{L}$ and $\pi_{R}$ defined in \eqref{Left Projection} and \eqref{Right Projection} are a fold and blowdown respectively. 
\end{theorem}

\section{Image Reconstruction}\label{sect:Imaging}
Next, we study the composition of $F$ with $F^{*}$. 
This composition is given as follows:
\begin{align*}
 F^{*}FT(x)=\int & e^{\I\lb \omega(t-\frac{1}{c_{0}}(|x-\g_{R}(s)|+|x-\g_{R}(s)|))- \wt{\omega}(t-\frac{1}{c_{0}}(|y-\g_{T}(s)|+|y-\g_{R}(s)|))\rb}\\
 &\times\overline{A(x,s,t,\omega)}A(y,s,t,\wt{\omega})\D s \D t\D \omega\D \wt{\omega}\D y
\end{align*}
After an application of the method of stationary phase, we can write the kernel of the operator $F^{*}F$ as 
\[
K(x,y)= \int e^{\I \frac{\omega}{c_{0}}\lb |y-\g_{T}(s)|+|y-\g_{R}(s)|-(|x-\g_{T}(s)|+|x-\g_{R}(s)|)\rb}\wt{A}(x,y,s,\omega)\D s \D \omega.
\]

Therefore the phase function of the kernel $K(x,y)$ is
\bel{def:phi}
\phi(x,y,s,\omega)=\frac{\omega}{c_{0}}\lb |y-\g_{T}(s)|+|y-\g_{R}(s)|-(|x-\g_{T}(s)|+|x-\g_{R}(s)|)\rb.
\ee 

\begin{proposition}\label{Wavefront set of Kernel}
\[
WF(K)'\subset \Delta\cup\Lambda
\]
where
$\Delta:=\{(x_{1},x_{2},\xi_{1},\xi_{2};x_{1},x_{2},\xi_{1},\xi_{2})\}$ and 
$\Lambda:=\{(x_{1},x_{2},\xi_{1},\xi_{2};x_{1},-x_{2},\xi_{1},-\xi_{2})\}$. Here for a point $x=(x_{1},x_{2})$, the covectors $(\xi_{1},\xi_{2})$ are non-zero multiples of the vector $(-\PD_{x_{1}}R(s,x),-\PD_{x_{2}}R(s,x))$, where $R$ is defined in \eqref{Bi-static distance}.
\end{proposition}
\begin{proof}
Using the H\"{o}rmander-Sato Lemma, we have 
\begin{equation*}
WF(K)'\subset
\end{equation*}
\begin{align}
\subset  
\notag\Bigg{\{}\NT&\Bigg{(}x_{1},x_{2},-\frac{\omega}{c_{0}}\NT\lb \frac{x_{1}-s-\A}{|x-\g_{T}(s)|}+\frac{x_{1}-s+\A}{|x-\g_{R}(s)|}\rb,-\frac{\omega}{c_{0}}\NT\lb \frac{x_{2}}{|x-\g_{T}(s)|}+\frac{x_{2}}{|x-\g_{R}(s)|}\rb\NT\NT\Bigg{)};\\
\notag &\NT\Bigg{(} y_{1},y_{2},-\frac{\omega}{c_{0}}\NT\lb \frac{y_{1}-s-\A}{|y-\g_{T}(s)|}+\frac{y_{1}-s+\A}{|y-\g_{R}(s)|}\rb,-\frac{\omega}{c_{0}}\NT\lb \frac{y_{2}}{|y-\g_{T}(s)|}+\frac{y_{2}}{|y-\g_{R}(s)|}\rb\NT\NT\Bigg{)}:\\
\notag &|x-\g_{T}(s)|+|x-\g_{R}(s)|=|y-\g_{T}(s)|+|y-\g_{R}(s)|,\\
\notag &\frac{x_{1}-s-\A}{|x-\g_{T}(s)|}+\frac{x_{1}-s+\A}{|x-\g_{R}(s)|}=\frac{y_{1}-s-\A}{|y-\g_{T}(s)|}+\frac{y_{1}-s+\A}{|y-\g_{R}(s)|},\quad\omega\neq 0\Bigg{\}}.
\end{align}
We now obtain a relation between $(x_{1},x_{2})$ and $(y_{1},y_{2})$. This is given by the following lemma.
\begin{lemma}
For all $s$, the set of all $(x_{1},x_{2})$, $(y_{1},y_{2})$ that satisfy
\begin{align}
 \label{Wavefront set relation Eqn-1} &|x-\g_{T}(s)|+|x-\g_{R}(s)|=|y-\g_{T}(s)|+|y-\g_{R}(s)|,\\%X_{1}+X_{2}=Y_{1}+Y_{2},\\
\label{Wavefront set relation Eqn-2}&\frac{x_{1}-s-\A}{|x-\g_{T}(s)|}+\frac{x_{1}-s+\A}{|x-\g_{R}(s)|}=\frac{y_{1}-s-\A}{|y-\g_{T}(s)|}+\frac{y_{1}-s+\A}{|y-\g_{R}(s)|}.
\end{align}
necessarily satisfy the following relations:  $x_{1}=y_{1}$ and $x_{2}=\pm y_{2}$. 
\end{lemma}
\begin{proof}
In order to show this,  we will consider \eqref{Wavefront set
relation Eqn-1} and \eqref{Wavefront set relation Eqn-2} as functions
of $\Rb^{3}$ by replacing $h$ in these expressions with $x_{3}-h$. We then transform these expressions using the coordinates \eqref{Prolate Spheroidal Coordinate System} and then set $x_{3}=y_{3}=0$ to prove the lemma.

 Consider the following change of coordinates:
\begin{equation}\label{Prolate Spheroidal Coordinate System}
\begin{array}{ll}
x_{1}=s+\A \cosh \rho \cos \theta & y_{1}=s+\A \cosh \rho' \cos \theta'\\
x_{2}=\A \sinh \rho \sin \theta \cos \vp & y_{2}=\A \sinh \rho' \sin \theta' \cos \vp'\\
x_{3}=h+ \A\sinh \rho \sin \theta \sin \vp & y_{3}=h+ \A\sinh \rho' \sin \theta' \sin \vp'
\end{array}
\end{equation}
where $s$, $\A>0$ and $h>0$ are fixed and $\rho\in[0,\infty)$,
$\theta\in[0,\pi]$ and $\vp\in[0,2\pi)$. This a well-defined
coordinate system except for $\xi=0$ and $\theta=0,\pi$. 

In the coordinate system \eqref{Prolate Spheroidal Coordinate System}, we have 
\begin{equation}\label{reln:x-gt}
\begin{array}{ll}
 |x-\g_{T}(s)|=\A(\cosh \rho-\cos\theta),\quad&
 |x-\g_{R}(s)|=\A(\cosh \rho+\cos\theta),\\
\frac{x_{1}-s-\A}{|x-\g_{T}(s)}=\frac{\cosh \rho\cos \theta-1}{\cosh \rho-\cos \theta},\quad&\frac{x_{1}-s+\A}{|x\g_{R}(s)|}=\frac{\cosh \rho\cos \theta+1}{\cosh \rho+\cos \theta}.
\end{array}
\end{equation}
The terms involving $y$ are obtained similarly.
Now \eqref{Wavefront set relation Eqn-1} and \eqref{Wavefront set relation Eqn-2} transform as follows:
\begin{align*}& 2\cosh \rho=2\cosh \rho'\\
& \frac{\cosh \rho\cos \theta-1}{\cosh \rho-\cos \theta}+\frac{\cosh \rho\cos \theta+1}{\cosh\rho+\cos \theta}=\frac{\cosh \rho'\cos \theta'-1}{\cosh \rho'-\cos \theta'}+\frac{\cosh \rho'\cos \theta'+1}{\cosh\rho'+\cos \theta'}.
\end{align*}
Using the first equality in the second equation, we have 
\[
\frac{\cos \theta}{\cosh^{2}\rho-\cos^{2}\theta}=\frac{\cos\theta'}{\cosh^{2}\rho-\cos^{2}\theta'}.
\]
This gives $\cos \theta=\cos \theta'$. Therefore $\theta=2n\pi\pm \theta'$, which then gives $\sin \theta=\pm \sin \theta'$. Therefore, in terms of $(x_{1},x_{2})$ and $(y_{1},y_{2})$, we have 
$x_{1}=y_{1}$ and $x_{2}=\pm y_{2}$. 
\end{proof}
Now to finish the proof of the proposition, when $x_{1}=y_{1}$ and $x_{2}=y_{2}$, there is contribution to $WF(K)'$ contained in the diagonal set $\Delta:=\{(x_{1},x_{2},\xi_{1},\xi_{2};x_{1},x_{2},\xi_{1},\xi_{2})\}$ and when $x_{1}=y_{1}$ and $x_{2}=-y_{2}$, we have a contribution to $WF(K)'$ contained in $\Lambda$, where 
$\Lambda:=\{(x_{1},x_{2},\xi_{1},\xi_{2};x_{1},-x_{2},\xi_{1},-\xi_{2})\}$.  
\end{proof}

 We use the following convention in the theorem below. The cotangent
variables corresponding to $x$ and $y$ are denoted as $\xi$ and $\eta$
respectively. Then note that $\xi=\PD_{x}\phi$ and
$\eta=-\PD_{y}\phi$. 

We now prove the following theorem: 
\begin{theorem}\label{Theorem-2}
Let $F$ be given by \eqref{Bi-static FIO}. Then $F^{*}F\in I^{3,0}(\Delta,\Lambda)$. 
\end{theorem}
\begin{proof}
We follow the proof of \cite[Theorem 1.6]{Felea} closely to prove this result. Recall that $\Delta$ and $\Lambda$ are defined by 
\begin{align*}
&\Delta=\{x_{1}-y_{1}=x_{2}-y_{2}=\xi_{1}-\eta_{1}=\xi_{2}-\eta_{2}=0\},\\
&\Lambda=\{x_{1}-y_{1}=x_{2}+y_{2}=\xi_{1}-\eta_{1}=\xi_{2}+\eta_{2}=0\}.
\end{align*}

The ideal of functions that vanish on $\Delta\cup \Lambda$ is generated by 
\begin{align*}
&\wt{p}_{1}=x_{1}-y_{1},\quad \wt{p}_{2}=x_{2}^{2}-y_{2}^{2},\quad \wt{p}_{3}=\xi_{1}-\eta_{1},\quad \wt{p}_{4}=(x_{2}+y_{2})(\xi_{2}-\eta_{2}),\\
&\wt{p}_{5}=(x_{2}-y_{2})(\xi_{2}+\eta_{2}),\quad \wt{p}_{6}=\xi_{2}^{2}-\eta_{2}^{2}.
\end{align*}
Let $p_{i}=q_{i}\wt{p}_{i}$, for $1\leq i\leq 6$, where $q_{1}, q_{2}$ are homogeneous of degree $1$ in $(\xi,\eta)$, $q_{3},q_{4}$ and $q_{5}$ are homogeneous of degree $0$ in $(\xi,\eta)$  and $q_{6}$ is homogeneous of degree $-1$ in $(\xi,\eta)$. Let $P_{i}$ be pseudodifferential operators with principal symbols $p_{i}$ for $1\leq i\leq 6$.

In order to prove that $F^{*}F\in I^{p,l}(\Delta,\Lambda)$, for some $p,l$, we have to show that $P_{i}K\in H^{s_{0}}_{\mathrm{loc}}$ for some $s_{0}$, for $1\leq i\leq 6$. By \cite[Proposition 4.3.1]{Duistermaat}, we have the following (up to lower order terms):
\[
P_{i}K(x,y)=\int e^{\I \frac{\omega}{c_{0}}\lb R(s,y)-R(s,x)\rb}\wt{A}(x,y,s,\omega)p_{i}(x,y,-\PD_{x}R(s,x),-\PD_{y}R(s,y))\D s\D \omega 
\]
We show in the Appendix that each $\wt{p}_{i}$ can be expressed in the following forms: 
\begin{align}
\label{Identity 1}&\wt{p}_{1}=\frac{f_{11}(x,y,s)}{\omega}\PD_{s}\phi+f_{12}(x,y,s)\PD_{\omega}\phi\\
\label{Identity 2}&\wt{p}_{2}=\frac{f_{21}(x,y,s)}{\omega}\PD_{s}\phi+f_{22}(x,y,s)\PD_{\omega}\phi\\
\label{Identity 3}&\wt{p}_{3}= -\PD_{s}\phi\\
\label{Identity 4}&\wt{p}_{4}=f_{41}(x,y,s)\PD_{s}\phi+\omega f_{42}(x,y,s)\PD_{\omega}\phi\\
\label{Identity 5}&\wt{p}_{5}=f_{51}(x,y,s)\PD_{s}\phi+\omega f_{52}(x,y,s)\PD_{\omega}\phi\\
\label{Identity 6}&\wt{p}_{6}=\omega f_{61}(x,y,s)\PD_{s}\phi+\omega^{2}f_{62}(x,y,s)\PD_{\omega}\phi.
\end{align}
Here $f_{ij}$ for $1\leq i\leq 6$ and $j=1,2$ are  smooth functions.

Therefore
\begin{align*}
P_{1}K(x,y)&=\int e^{\I \phi(x,y,s,\omega)}\wt{A}(x,y,s,\omega)q_{1}\lb \frac{f_{11}(x,y,s)}{\omega}\PD_{s}\phi+f_{12}(x,y,s)\PD_{\omega}\phi\rb \D s\D \omega\\
&=\int \PD_{s}\lb e^{\I \phi(x,y,s,\omega)}\rb\frac{q_{1}}{\I \omega}\wt{A}(x,y,s,\omega)f_{11}(x,y,s)\D s\D \omega\\
&+\int \PD_{\omega}\lb e^{\I \phi(x,y,s,\omega)}\rb\frac{q_{1}}{\I}\wt{A}(x,y,s,\omega)f_{12}(x,y,s)\D s\D \omega\\
\intertext{By integration by parts}
&=-\Bigg{\{}\int e^{\I \phi(x,y,s,\omega)}\PD_{s}\lb \frac{q_{1}}{\I \omega}\wt{A}(x,y,s,\omega)f_{11}(x,y,s)\rb \D s\D \omega\\
&+\int e^{\I \phi(x,y,s,\omega)}\PD_{\omega}\lb\frac{q_{1}}{\I}\wt{A}(x,y,s,\omega)f_{12}(x,y,s)\rb \D s\D \omega \Bigg{\}}.
\end{align*}
Note that $q_{1}$ is homogeneous of degree $1$ in $\omega$, and $\wt{A}$ is a symbol of order $4$, hence each amplitude term in the sum above is of order $4$.

Therefore by \cite[Theorem 2.2.1]{Duistermaat}, we have that $P_{1}K \in H^{s_{0}}_{\mathrm{loc}}$.

A similar argument works for each of the other five pseudodifferential operators. Hence by  \cite[Proposition 1.35]{GU1990a}, we have that $F^{*}F\in I^{p,l}(\Delta,\Lambda)$. Because $C$ is a local canonical graph away from $\Sigma$, the transverse intersection calculus applies for the composition $F^{*}F$ away from $\Sigma$. Hence $F^{*}F$ is of order $3$ on $\Delta\setminus \Sigma$ and $\Lambda\setminus \Sigma$. Since $F^{*}F$ is of order $p+l$ on $\Delta\setminus \Sigma$ and is of order $p$ on $\Lambda\setminus \Sigma$, we have that $p=3$ and $l=0$. Therefore the theorem is proved.

\end{proof}

\section*{Acknowledgments} The first named author
thanks Dave Isaacson, Margaret Cheney, Birsen Yaz{\i}c{\i} and Art
Weiss for discussions regarding this work while he was a post-doctoral fellow at RPI. Additionally, he thanks the
Department of Mathematics at Tufts University for providing a
wonderful research environment. Both authors thank Cliff Nolan, Raluca
Felea, Allan Greenleaf, and Gunther Uhlmann for stimulating
discussions about mathematics related to this research. 

\begin{appendix}\label{sect:Appendix}

\section{}

Here we prove the identities \eqref{Identity 1} through \eqref{Identity 6}
that are required in the proof of Theorem \ref{Theorem-2}.  For
convenience, and without loss of generality, we will assume $c_0 = 1$.

In obtaining these identities, it is easier to work in the coordinate
system defined in \eqref{Prolate Spheroidal Coordinate System}. We
will work with the extension $\wt{\phi}$ of the phase function $\phi$
to $\Rb^{3}$ defined by

\begin{align*}\label{Extended phase function}
\wt{\phi}=\frac{\omega}{c_{0}}&\Bigg{(}\sqrt{(y_{1}-s-\A)^{2}+y_{2}^{2}+(y_{3}-h)^{2}}+\sqrt{(y_{1}-s+\A)+y_{2}{2}+(y_{3}-h)^{2}}\\
&-\lb\sqrt{(x_{1}-s-\A)^{2}+x_{2}^{2}+(x_{3}-h)^{2}}+\sqrt{(x_{1}-s+\A)+x_{2}^{2}+(x_{3}-h)^{2}}\rb\Bigg{)}.
\end{align*}
Then, using the facts that
\begin{equation}\label{w-Der of Extended Phase}
\PD_{\omega}\wt{\phi}\vert_{x_{3}=y_{3}=0}=\PD_{\omega}\phi\quad \mbox{ and }\quad  \PD_{s}\wt{\phi}\vert_{x_{3}=y_{3}=0}=\PD_{s}\phi,
\end{equation}
 we set the third coordinate $x_{3}=y_{3}=0$  to obtain the required identities. 

\subsection{Expression for $x_{1}-y_{1}$}

We now obtain an expression for $x_{1}-y_{1}$ of the form 
\begin{equation}\label{Expression-1}
x_{1}-y_{1}=\frac{f_{11}(x,y,s)}{\omega}\PD_{s}\phi+f_{12}(x,y,s)\PD_{\omega}\phi,
\end{equation}
where $f_{11}$ and $f_{12}$ are smooth functions.

That is,  denoting $A_{1}=x_{1}-y_{1}$, we would like to obtain an expression of the form \eqref{Expression-1} involving $\PD_{s}\phi$ and $\PD_{\omega}\phi$ for 
\begin{equation}\label{Expression-1 in prolate}
A_{1}=\A(\cosh \rho\cos \theta-\cosh\rho'\cos\theta').
\end{equation} 

We have 
\[
\PD_{\omega}\wt{\phi}=2\A(\cosh \rho'-\cosh \rho)
\]
and

\begin{align}
\notag \PD_{s}\wt{\phi}=&-{\omega}\Bigg{(} \lb\frac{\cosh \rho'\cos \theta'-1}{\cosh \rho'-\cos \theta'}+\frac{\cosh \rho'\cos \theta'+1}{\cosh\rho'+\cos \theta'}\rb\\
\notag &-\left(\frac{\cosh \rho\cos \theta-1}{\cosh \rho-\cos \theta}+\frac{\cosh \rho\cos \theta+1}{\cosh\rho+\cos \theta}\right)\Bigg{)}\\
\label{s-Der of phi in prolate}&={2\omega}\lb \frac{\sinh^{2}\rho\cos \theta}{\cosh^{2}\rho-\cos^{2}\theta}- \frac{\sinh^{2}\rho'\cos
\theta'}{\cosh^{2}\rho'-\cos^{2}\theta'}\rb
\end{align}
using $\cosh^2\rho\cos\theta - \cos \theta = \sinh^2\rho \cos\theta$.  Now 

\begin{equation}\label{Eqn in A and B} \cos \theta \frac{\PD_{\omega}\wt{\phi}}{2}=\A \lb \cosh \rho' \cos \theta'-\cosh\rho\cos\theta\rb+\A \cosh\rho'\lb \cos \theta-\cos \theta'\rb.
\end{equation}
Then
\bel{equn:A1}
A_{1}=-\frac{\cos\theta}{2}\PD_{\omega}\wt{\phi}+\A\cosh\rho'(\cos \theta-\cos\theta').
\ee
Adding and subtracting
$\frac{\sinh^{2}\rho\cos\theta'}{\cosh^{2}\rho-\cos^{2}\theta'}
$  inside \eqref{s-Der of phi in prolate}, we have
\begin{align*}
\PD_{s}\wt{\phi}&={2\omega} \Bigg{(}\underbrace{\frac{\sinh^{2}\rho\cos\theta}{\cosh^{2}\rho-\cos^{2}\theta}-\frac{\sinh^{2}\rho\cos\theta'}{\cosh^{2}\rho-\cos^{2}\theta'}}_{I}\\
&+\underbrace{\frac{\sinh^{2}\rho\cos\theta'}{\cosh^{2}\rho-\cos^{2}\theta'}-\frac{\sinh^{2}\rho'\cos\theta'}{\cosh^{2}\rho'-\cos^{2}\theta'}}_{II}\Bigg{)}.
\end{align*}
Simplifying $I$ and $II$, we have,
\begin{align*}
I&= \frac{(\cos\theta-\cos\theta')(\sinh^{2}\rho)(\cosh^{2}\rho+\cos\theta\cos\theta')}{(\cosh^{2}\rho-\cos^{2}\theta)(\cosh^{2}\rho-\cos^{2}\theta')}
\end{align*}
and 
\begin{align*}
II&=\frac{\cos\theta'\sin^{2}\theta'(\cosh \rho-\cosh\rho')(\cosh \rho+\cosh \rho')}{(\cosh^{2}\rho'-\cos^{2}\theta')(\cosh^{2}\rho-\cos^{2}\theta')}\\
&= -\frac{\cos\theta'\sin^{2}\theta'(\cosh\rho+\cosh\rho')}{(\cosh^{2}\rho'-\cos^{2}\theta')(\cosh^{2}\rho-\cos^{2}\theta')}\cdot \frac{\PD_{\omega}\wt{\phi}}{2\A}
\end{align*} 
where we have used the fact $\sinh^2\rho - \sinh^2 \rho' =
\cosh^2\rho- \cosh^2\rho'$.

Using these calculations, we see
\begin{align}\label{Expression for B}
\cos\theta-\cos\theta'=&\lb \frac{(\cosh^{2}\rho-\cos^{2}\theta)(\cosh^{2}\rho-\cos^{2}\theta')}{(\cosh^{2}\rho-1)(\cosh^{2}\rho+\cos\theta\cos\theta')}\rb\\
\notag &\times\lb \frac{\PD_{s}\wt{\phi}}{2\omega}+\frac{\cos\theta'\sin^{2}\theta'(\cosh\rho+\cosh\rho')}{(\cosh^{2}\rho'-\cos^{2}\theta')(\cosh^{2}\rho-\cos^{2}\theta')}\cdot \frac{\PD_{\omega}\wt{\phi}}{2\A}\rb.
\end{align}
Now setting $x_{3}=y_{3}=0$ and using \eqref{equn:A1} , we have,
\begin{align*}
A_{1}=&\frac{\A\cosh\rho'(\cosh^{2}\rho-\cos^{2}\theta)(\cosh^{2}\rho-\cos^{2}\theta')}{(\cosh^{2}\rho-1)(\cosh^{2}\rho+\cos\theta\cos\theta')}\frac{\PD_{s}\phi}{2\omega}\\
& -\frac{1}{2}\lb \cos\theta-\frac{\cosh\rho'\cos\theta'\sin^{2}\theta'(\cosh \rho+\cosh\rho')}{(\cosh^{2}\rho-1)(\cosh^{2}\rho+\cos\theta\cos\theta')}\cdot \frac{(\cosh^{2}\rho-\cos^{2}\theta)}{(\cosh^{2}\rho'-\cos^{2}\theta')}\rb \PD_{\omega}\phi.
\end{align*}
We can see that no denominator in this expression is zero for $x_3 =
0$ (since $\cosh \rho > 1\geq\cos\theta$ if $x_3=0$) and so this
expression for $A_1$ is defined and smooth for all values of the
coordinates.

We can write $A_{1}$ in the Cartesian coordinate system as follows.
First, for simplicity, let
\begin{align}
\label{def:X1} &X_{1}=|x-\g_{T}(s)|=\sqrt{(x_{1}-s-\A)^{2}+x_{2}^{2}+h^{2}},\\
\label{def:X2}&X_{2}=|x-\g_{R}(s)|=\sqrt{(x_{1}-s+\A)^{2}+x_{2}^{2}+h^{2}},\\
\label{def:Y1}&Y_{1}=|y-\g_{T}(s)|=\sqrt{(y_{1}-s-\A)^{2}+y_{2}^{2}+h^{2}},\\
\label{def:Y2}&Y_{2}=|y-\g_{R}(s)|=\sqrt{(y_{1}-s+\A)^{2}+y_{2}^{2}+h^{2}}.
\end{align}

Then using these expressions and \eqref{reln:x-gt} we see that
\begin{align*}
A_{1}&= \frac{\A\lb\frac{Y_{1}+Y_{2}}{2\A}\rb\lb \frac{X_{1}X_{2}}{\A^{2}}\rb\lb \frac{(X_{1}+X_{2})^{2}}{4\A^{2}}-\frac{4(y_{1}-s)^{2}}{(Y_{1}+Y_{2})^{2}}\rb}
{2\omega\left( \frac{(X_{1}+X_{2})^{2}}{4\A^{2}}-1\right)\left(\frac{(X_{1}+X_{2})^{2}}{4\A^{2}}+\frac{4(x_{1}-s)(y_{1}-s)}{(X_{1}+X_{2})(Y_{1}+Y_{2})}\right)}\PD_{s}\phi\\
&\
 -\frac{1}{2}\left(\frac{2(x_{1}-s)}{X_{1}+X_{2}}-\frac{X_{1}X_{2}\left(\frac{2(y_{1}-s)}{Y_1+Y_2}\right)\left(1-\frac{4(y_{1}-s)^{2}}{(Y_{1}+Y_{2})^{2}}\right)
\lb\frac{X_{1}+X_{2}+Y_{1}+Y_{2}}{2\A}\rb}{Y_{1}Y_{2}\left(
\frac{(X_{1}+X_{2})^{2}}{4\A^{2}}-1\right)\left(\frac{(X_{1}+X_{2})^{2}}{4\A^2}+\frac{4(x_{1}-s)(y_{1}-s)}{(X_{1}+X_{2})(Y_{1}+Y_{2})}\right)}\right)\PD_{\omega}\phi.
\end{align*} We see from this second expression for $A_1$ that the
coefficient functions are smooth in $(x,y,s)$ since the expressions
\eqref{def:X1}-\eqref{def:Y2} are non-zero and smooth. 

\subsection{Expression for $x_{2}^{2}-y_{2}^{2}$}

Now we write $x_{2}^{2}-y_{2}^{2}$ in the form 
\begin{equation}\label{A2}
A_{2}:=x_{2}^{2}-y_{2}^{2}=\frac{f_{21}(x,y,s)}{\omega}\PD_{s}\phi+f_{22}(x,y,s)\PD_{\omega}\phi,
\end{equation}
where $f_{21}$ and $f_{22}$ are smooth functions.
  $A_{2}$ in the coordinate system \eqref{Prolate Spheroidal Coordinate System} is
 \begin{align}  
 \notag A_{2}:&=\A^{2}(\sinh^{2}\rho\sin^{2}\theta\cos^{2}\vp-\sinh^{2}\rho'\sin^{2}\theta'\cos^{2}\vp')\\
\label{first x2-y2 term}  &=\A^{2}(\sinh^{2}\rho\sin^{2}\theta-\sinh^{2}\rho'\sin^{2}\theta')\\
\label{second x2-y2 term} &+\A^{2}(\sinh^{2}\rho'\sin^{2}\theta'\sin^{2}\vp'-\sinh^{2}\rho\sin^{2}\theta\sin^{2}\vp).
 \end{align}
For $x_{3}=y_{3}=0$, \eqref{second x2-y2 term} is $0$.
So it is enough to obtain an expression of the form \eqref{A2} for \eqref{first x2-y2 term}, which we still denote by $A_{2}$.

Using the following identities, 
\[
\sinh^{2} \rho = 
\cosh^{2}\rho-1 \mbox{ and } \sin^{2}\theta=1-\cos^{2}\theta,
\]
we have,
\begin{align*}
A_{2}&=\A^{2}(\sinh^{2}\rho\sin^{2}\theta-\sinh^{2}\rho'\sin^{2}\theta')\\
&=\A^{2}\lb (\cosh^{2}\rho-\cosh^{2}\rho')+(\cos^{2}\theta-\cos^{2}\theta')-(\cosh^{2}\rho\cos^{2}\theta-\cosh^{2}\rho'\cos^{2}\theta')\rb\\
&= -\frac{\A}{2}(\cosh\rho+\cosh\rho')\PD_{\omega}\wt{\phi}+\A^{2}(\cos\theta+\cos\theta')(\cos\theta-\cos\theta')\\
&-\A(\cosh\rho\cos\theta+\cosh\rho'\cos\theta')A_{1}\\
&= -\frac{\A}{2}(\cosh\rho+\cosh\rho')\PD_{\omega}\wt{\phi}+\A^{2}(\cos\theta+\cos\theta')(\cos \theta-\cos\theta')\\
&-\A(\cosh\rho\cos\theta+\cosh\rho'\cos\theta')( -\frac{\cos\theta}{2}\PD_{\omega}\phi+\A \cosh\rho' (\cos\theta-\cos\theta'))\\
&=\frac{\A}{2}\Big{(} -(\cosh \rho+\cosh \rho')+(\cosh\rho\cos\theta+\cosh\rho'\cos\theta')\cos\theta\Big{)}\PD_{\omega}\wt{\phi}\\
&+\A^{2}\Big{(}(\cos\theta+\cos\theta')
-(\cosh\rho\cos\theta+\cosh\rho'\cos\theta')\cosh
\rho'\Big{)}(\cos\theta-\cos\theta').
\end{align*}
Now we use the expression for
$\cos\theta-\cos\theta'$ in Equation \eqref{Expression for B} and set
$x_{3}=y_{3}=0$; this shows that $x_{2}^{2}-y_{2}^{2}$ can be written
in the form 
\begin{equation*}
A_{2}=\frac{f_{21}(x,y,s)}{\omega}\PD_{s}\phi+f_{22}(x,y,s)\PD_{\omega}\phi.
\end{equation*}
\subsection{Expression for $\xi_{1}-\eta_{1}$} Now we consider
$\xi_{1}-\eta_{1}$, where we recall that $(\xi_{1},\xi_{2}) =
\PD_{x}\phi$ and $(\eta_{1},\eta_{2})=-\PD_{y}\phi$ are the cotangent
variables corresponding to $(x_{1},x_{2})$ and $(y_{1},y_{2})$
respectively. 

Then note that $\xi_{1}-\eta_{1}$ is $\PD_{x_{1}}\phi+\PD_{y_{1}}\phi$. But this is the same as $-\PD_{s}\phi$. 
Hence 
\[
A_{3}:=\xi_{1}-\eta_{1}=-\PD_{s}\phi.
\]

\subsection{Expression for $(x_{2}-y_{2})(\xi_{2}+\eta_{2})$}

We have (up to a negative sign)
\begin{align*}
(x_{2}-y_{2})(\xi_{2}+\eta_{2})&=\omega(x_{2}-y_{2})\lb
\frac{x_{2}}{\norm{x-\g_{T}}}+\frac{x_{2}}{|x-\g_{R}|}+\frac{y_{2}}{\norm{y-\g_{T}}}+\frac{y_{2}}{\norm{y-\g_{R}}}\rb\\
&=
2\omega\Bigg{(} \frac{x_{2}^{2}\cosh \rho}{\cosh^{2}\rho-\cos^{2}\theta}-\frac{y^{2}\cosh \rho'}{\cosh^{2}\rho'-\cos^{2}\theta'}\\
&\hspace{0.3in}+\frac{x_{2}y_{2}\cosh\rho'}{\cosh^{2}\rho'-\cos^{2}\theta'}-\frac{x_{2}y_{2}\cosh \rho}{\cosh^{2}\rho-\cos^{2}\theta}\Bigg{)}\\
&=2\omega\Bigg{(} \frac{x_{2}^{2}\cosh \rho}{\cosh^{2}\rho-\cos^{2}\theta}-\frac{x_{2}^{2}\cosh \rho'}{\cosh^{2}\rho'-\cos^{2}\theta'}\\
&+(x_{2}^{2}-y_{2}^{2})\frac{\cosh \rho'}{\cosh^{2}\rho'-\cos^{2}\theta'}\\
&+
\frac{x_{2}y_{2}\cosh\rho'}{\cosh^{2}\rho'-\cos^{2}\theta'}-\frac{x_{2}y_{2}\cosh \rho}{\cosh^{2}\rho-\cos^{2}\theta}\Bigg{)},\\ 
\end{align*}
Here we have added and subtracted $\frac{x_{2}^{2}\cosh \rho'}{\cosh^{2}\rho'-\cos^{2}\theta'}$ in the previous equation. Simplifying this we get,
\begin{align*}
(x_{2}-y_{2})(\xi_{2}+\eta_{2})=&
2\omega\Bigg{(}(x_{2}^{2}-x_{2}y_{2})\Bigg[
\frac{(\cosh \rho\cosh \rho'+\cos^{2}\theta)(\cosh \rho'-\cosh \rho)}{(\cosh^{2}\rho-\cos^{2}\theta)(\cosh^{2}\rho'-\cos^{2}\theta')}\\
&\quad\quad\phantom{(x_{2}^{2}-x_{2}y_{2})}+\frac{\cosh \rho(\cos\theta+\cos\theta')(\cos
\theta-\cos\theta')}{(\cosh^{2}\rho-\cos^{2}\theta)(\cosh^{2}\rho'-\cos^{2}\theta')}\bigg]\\
&+(x_{2}^{2}-y_{2}^{2})\frac{\cosh \rho'}{\cosh^{2}\rho'-\cos^{2}\theta'}\Bigg{)}.
\end{align*}

Now note that  $\cosh \rho'-\cosh \rho=\frac{\PD_{\omega}\phi}{2\A}$ and we already have expressions for $\cos\theta-\cos\theta'$ and $x_{2}^{2}-y_{2}^{2}$ involving combinations of $\PD_{\omega}\phi$ and $\PD_{s}\phi$. 

Hence we can write $(x_{2}-y_{2})(\xi_{2}+\eta_{2})$ in the form
\[
(x_{2}-y_{2})(\xi_{2}+\eta_{2})=f_{41}(x,y,s)\PD_{s}\phi+\omega f_{42}(x,y,s)\PD_{\omega}\phi.
\]

For future reference, note that our calculation in this section shows
that \bel{handy-id}\begin{aligned}\frac{\cosh
\rho}{\cosh^{2}\rho-\cos^{2}\theta}-\frac{\cosh
\rho'}{\cosh^{2}\rho'-\cos^{2}\theta'} &= \frac{(\cosh \rho\cosh
\rho'+\cos^{2}\theta)(\cosh \rho'-\cosh
\rho)}{(\cosh^{2}\rho-\cos^{2}\theta)(\cosh^{2}\rho'-\cos^{2}\theta')}\\
&\qquad+\frac{\cosh \rho(\cos\theta+\cos\theta')(\cos
\theta-\cos\theta')}
{(\cosh^{2}\rho-\cos^{2}\theta)(\cosh^{2}\rho'-\cos^{2}\theta')}\end{aligned}\ee

\subsection{Expression for $(x_{2}+y_{2})(\xi_{2}-\eta_{2})$}
This is very similar to the derivation of the expression we obtained for $(x_{2}-y_{2})(\xi_{2}+\eta_{2})$.
\subsection{Expression for $\rho_{2}^{2}-\eta_{2}^{2}$}

Using \eqref{def:X2} and \eqref{def:Y2}, we have 
\begin{align*}
\xi_{2}^{2}-\eta_{2}^{2}&=\omega^{2}\lb
\lb\frac{x_{2}}{|x-\g_{T}|}+\frac{x_{2}}{|x-\g_{R}|}\rb^2-\lb\frac{y_{2}}{|y-\g_{T}|}+\frac{y_{2}}{|y-\g_{R}|}\rb^2\rb\\
&={4}\omega^{2}\lb x_{2}^{2}\frac{\cosh^{2}\rho}{(\cosh^{2}\rho-\cos^{2}\theta)^{2}}-y_{2}^{2}\frac{\cosh^{2}\rho'}{(\cosh^{2}\rho'-\cos^{2}\theta')^{2}}\rb\\
&={4}\omega^{2}\Bigg{\{} x_{2}^{2}\lb
\frac{\cosh^{2}\rho}{(\cosh^{2}\rho-\cos^{2}\theta)^{2}}-\frac{\cosh^{2}\rho'}{(\cosh^{2}\rho'-\cos^{2}\theta')^{2}}\rb\\
&+(x_{2}^{2}-y_{2}^{2})\frac{\cosh^{2}\rho'}{(\cosh^{2}\rho'-\cos^{2}\theta')^{2}}\Bigg{\}}.
\end{align*}
Now using the computations for $x_{2}^{2}-y_{2}^{2}$ and
$(x_{2}-y_{2})(\rho_{2}+\eta_{2})$, in particular \eqref{handy-id},
 we can write $\xi_{2}^{2}-\eta_{2}^{2}$ in the form
\[
\xi_{2}^{2}-\eta_{2}^{2}=\omega f_{61}(x,y,s)\PD_{s}\phi+\omega^{2}f_{62}(x,y,s)\PD_{\omega}\phi
\]
for smooth functions $f_{61}, f_{62}$.

\end{appendix}

\medskip

%Received \emph{To be filled} ; revised \emph{To be filled}. 

\medskip
 {\it E-mail address: }{\tt venkyp.krishnan@gmail.com}\\
 \indent{\it E-mail address: }{\tt todd.quinto@tufts.edu}\\

\begin{thebibliography}{99}

\bibitem{And}
\newblock L-E.\ Andersson, 
\newblock \emph{{On the determination of a function from spherical
  averages}}, 
\newblock SIAM J. Math. Anal. 
\newblock \textbf{19} (1988), 214--232.



\bibitem{Cheney2001}
     \newblock M. Cheney,
     \newblock \emph{A mathematical tutorial on synthetic aperture radar},
     \newblock SIAM Rev., \textbf{43}
     (2001), 301 -- 312. 

\bibitem{CheneyBordenBook}
    \newblock M. Cheney and B. Borden,
    \newblock ``Fundamentals of radar imaging,"
    \newblock Society for Industrial and Applied Mathematics (SIAM), Philadelphia, Pennsylvania, 2009.
    \bibitem{Duistermaat}
    \newblock J. Duistermaat,
    \newblock ``Fourier {I}ntegral {O}perators,"
    \newblock Birkh\"auser Boston Inc., Boston, Massachusetts, 1996.

 

\bibitem{CB1979}
\newblock J.~Cohen and H.~Bleistein, 
\newblock \emph{{Velocity inversion procedure for acoustic
  waves}},
\newblock  Geophysics \textbf{44} (1979), 1077--1085.

 
    \bibitem{Felea}
     \newblock R. Felea,
     \newblock \emph{Composition of {F}ourier integral operators with fold and blowdown singularities},
     \newblock Comm. Partial Differential Equations, \textbf{30}
     (2005), 1717--1740. 

\bibitem{RF2}
\newblock R. Felea,
\newblock {Displacement of artefacts in inverse scattering},
\newblock {Inverse Problems},
\textbf{23}, (2007), 1519--1531.


\bibitem{FG}
\newblock R. Felea and A. Greenleaf, 
\newblock \emph{{An FIO calculus for
marine seismic
imaging: folds and crosscaps}},  
\newblock Comm. P.D.E. {\bf 33} (2008), 45--77.


\bibitem{GU1989}
\newblock A. Greenleaf and G. Uhlmann, 
\newblock \emph{{Non-local inversion formulas for
  the X-ray transform}}, 
\newblock Duke Math. J. \textbf{58} (1989), 205--240.

\bibitem{GU1990a} 
\newblock  A. Greenleaf and G. Uhlmann, 
\newblock \emph{Estimates for
singular Radon transforms and pseudodifferential operators with
singular symbols}, 
\newblock J.  Functional Anal.  \textbf{89} (1990), 202-232.

\bibitem{GU1990b} 
\newblock A. Greenleaf and G. Uhlmann, 
\newblock \emph{{Composition of
some singular Fourier integral operators and estimates for restricted
X-ray transforms}, 
\newblock Ann. Inst. Fourier, Grenoble}, \textbf{40} (1990),
443--466.  

\bibitem{GU1990c} 
\newblock A.\ Greenleaf and G.\ Uhlmann, 
\newblock \emph{Microlocal
techniques in integral geometry}, 
\newblock  Contemporary Math., \textbf{113}
(1990), 121-136.



\bibitem{G} \newblock V. Guillemin,  
\newblock \emph{{ Cosmology in $(2 +
1)$-dimensions, cyclic models, and deformations of $M\sb {2,1}$}},
\newblock Annals of Mathematics Studies, {\bf 121}.  Princeton University Press,
Princeton, NJ, 1989.

\bibitem{GS1977} 
\newblock V. Guillemin and S. Sternberg, \emph{{Geometric
Asymptotics}}, 
\newblock American Mathematical Society, Providence, RI, 1977.

\bibitem{GUh} 
\newblock V. Guillemin and G.  Uhlmann, 
\newblock \emph{{ Oscillatory integrals
with singular symbols}} 
\newblock Duke Math. J.  1981, {\bf 48} (1), 251-267.

\bibitem{Ho1971} 
\newblock L. H{\"o}rmander, 
\newblock \emph{{Fourier integral operators,
I}}, 
\newblock Acta Mathematica \textbf{127} (1971), 79--183.

\bibitem{HorneYates}
\newblock A. M. Horne and G. Yates,
\newblock \emph{Bistatic synthetic aperture radar}
\newblock Proc. IEEE Radar Conf., 2002, 6--10.


\bibitem{Ho} 
\newblock L. H{\"o}rmander, 
\newblock \emph{The Analysis of Linear Partial
Differential Operators, {I --IV}}, 
\newblock Springer Verlag, New York, 1983.


\bibitem{LQ2000}
\newblock A.K. Louis and E.T. Quinto, 
\newblock\emph{{Local Tomographic Methods in
  SONAR}}, 
\newblock Surveys on Solution Methods for Inverse Problems (Vienna/New York)
  (D.Colton, H.Engl, A.Louis, J.McLaughlin, and W.~Rundell, eds.), Springer
  Verlag, 2000, pp.~147--154.

\bibitem{MU} 
\newblock R. Melrose, G. Uhlmann, 
\newblock \emph{{Lagrangian
intersection and the Cauchy problem}} 
\newblock Comm. Pure Appl. Math., 1979,
{\bf 32} (4), 483--519.

\bibitem{NC2002} \newblock C. J. Nolan and M. Cheney,
\newblock\emph{Synthetic aperture inversion}, Inverse Problems,
\textbf{18}, 2002, 221--235.

\bibitem{NC2004} \newblock C. J. Nolan and M. Cheney,
\newblock\emph{Microlocal analysis of synthetic aperture radar imaging}, J. Fourier Anal. Appl.,
\textbf{10}, 2004, 133--148.



\bibitem{Q1980}
\newblock E.T. Quinto, 
\newblock \emph{{The dependence of the generalized Radon transform on
defining measures}}, 
\newblock Trans. Amer. Math. Soc. \textbf{257} (1980), 331--346.

\bibitem{Treves}
    \newblock F. Tr\'eves,
    \newblock ``Introduction to pseudodifferential and {F}ourier integral
              operators. {V}ol. 1 and 2,"
    \newblock Plenum Press New York 1980. 
\bibitem{YarmanYaziciCheney}
\newblock C. E. Yarman, B. Yaz{\i}c{\i} and M. Cheney,
\newblock \emph{Bistatic synthetic aperture radar imaging with arbitrary trajectories},
\newblock IEEE Transactions on Image Processing, \textbf{17}, (2008), 84--93.

\end{thebibliography}
  \end{document}